\RequirePackage{ifluatex}

\documentclass{article}

\usepackage{amsmath}
\usepackage{amssymb}
\usepackage{amsthm}
\usepackage{mathtools}
\usepackage[latin1]{inputenc}
     
\usepackage{graphicx}

\usepackage[noblocks]{authblk}
\usepackage[usenames,dvipsnames]{pstricks}
\usepackage{epsfig}
\usepackage{pst-grad} 
\usepackage{pst-plot} 

\usepackage{ifthen}
\usepackage{color}

\newcommand{\intav}[1]{\mathchoice {\mathop{\vrule width 6pt height 3 pt depth  -2.5pt
\kern -8pt \intop}\nolimits_{\kern -6pt#1}} {\mathop{\vrule width
5pt height 3  pt depth -2.6pt \kern -6pt \intop}\nolimits_{#1}}
{\mathop{\vrule width 5pt height 3 pt depth -2.6pt \kern -6pt
\intop}\nolimits_{#1}} {\mathop{\vrule width 5pt height 3 pt depth
-2.6pt \kern -6pt \intop}\nolimits_{#1}}}

\def\polhk#1{\setbox0=\hbox{#1}{\ooalign{\hidewidth\lower1.5ex\hbox{`}\hidewidth\crcr\unhbox0}}}

\def\XXint#1#2#3{{\setbox0=\hbox{$#1{#2#3}{\int}$ }
\vcenter{\hbox{$#2#3$ }}\kern-.6\wd0}}

\newcommand{\osc}{\operatorname{osc}}

\newtheorem{Theorem}{Theorem}[section]

\newtheorem{Definition}{Definition}[section]
\newtheorem{Lemma}{Lemma}[section]

\newtheorem{Proposition}{Proposition}[section]
\newtheorem{Remark}{Remark}[section]
\newtheorem{Assumption}{A}

\usepackage{fullpage}

\begin{document}


\title{Regularity estimates for fully nonlinear integro-differential equations with nonhomogeneous degeneracy}

\author[1]{P\^edra D. S. Andrade
\footnote{pedra.andrade@tecnico.ulisboa.pt}}
\author[2]{Disson S. dos Prazeres
\footnote{disson@mat.ufs.br}}
\author[3]{Makson S. Santos
\footnote{ makson.santos@tecnico.ulisboa.pt}}	
\affil[1]{\scriptsize Instituto Superior T\'ecnico, Universidade de Lisboa}
\affil[2]{\scriptsize Departamento de Matem\'atica, Universidade Federal de Sergipe}
\affil[3]{\scriptsize Instituto Superior T\'ecnico, Universidade de Lisboa}
\maketitle

\begin{abstract} We investigate the regularity of the solutions for a class of degenerate/singular fully nonlinear nonlocal equations.  In the degenerate scenario, we  establish that there exists at least one viscosity solution of class $C_{loc}^{1, \alpha}$,  for some constant $\alpha \in (0, 1)$. In addition, under suitable conditions on $\sigma$, we prove regularity estimates in H\"older spaces for any viscosity solution.  We also examine the singular setting and prove H\"older regularity estimates for the gradient of the solutions.
\noindent    
\medskip


\medskip

\end{abstract}

\vspace{.1in}

\section{Introduction}
In this paper, we examine the regularity of solutions to degenerate/singular fully nonlinear integro-differential equations of the form
\begin{equation}\label{main_eq0}
	- \left(|Du(x)|^p + a(x)|Du(x)|^q \right){\mathcal I}_{\sigma}(u, x)= f(x) \quad \text{in} \quad B_1, 
\end{equation}
where $0\leq a(\cdot) \in C(B_1)$, ${\mathcal I}_{\sigma}$ is a fully nonlinear elliptic integro-differential operator with order $\sigma \in(1, 2)$ and $f \in L^{\infty}(B_1)$. We establish H\"older regularity estimates for the gradient of bounded viscosity solutions to \eqref{main_eq0}. More precisely, in the degenerate setting $0 < p \leq q$, we show that for any $1<\sigma<2$, there exists at least one viscosity solution $u$ to \eqref{main_eq0}, so that $u \in C_{loc}^{1,\alpha}(B_1)$ for some $\alpha \in (0,1)$. This result is new even in the case $a \equiv 0$. Moreover, we prove the existence of a constant $\sigma_0$, sufficiently close to $2$, for which solutions are of class $C_{loc}^{1,\alpha}$ for some $\alpha \in (0,1)$, provided $\sigma_0<\sigma <2$. In the singular framework $-1 < p \leq q <0$,  we show that solutions belong to $C_{loc}^{1,\alpha}(B_1)$ for some $\alpha \in (0,1)$. 

In the literature, integro-differential operators are commonly referred to as nonlocal operators. This class of problems appears in many mathematical modeling processes, such as image processing and payoff models, see \cite{Bisci-Radulescu-Servadei-2016, Bucur-Valdinoci2016,Nezza-Palatucci-Valdinoci-2012}, to mention just a few. For instance, linear nonlocal operators arise naturally in discontinuous stochastic processes taking the form
\begin{equation}\label{Eq-integral operator}
	Lu(x):=\sum_{i,j}a_{ij}\partial_{ij} u + \sum_{i}b_i\partial_i u + \int_{\mathbb{R}^d}\left(u(x+y) -u(x) - Du\cdot y\chi_{B_1}(y)\right)d\mu(y), 
\end{equation}
for a suitable measure $\mu$. Meanwhile, nonlinear integro-differential equations can be found in stochastic control problems, where a player can change strategies at every step. In this case, we end up with convex nonlinear equations driven by the operator
\[
Iu(x) := \sup_\alpha L_\alpha u(x).
\]
In addition, we can also obtain models of the type
\[
Iu(x) := \inf_\beta\sup_{\alpha} L_{\alpha,\beta} u(x),
\]
whenever two or more players are involved, see \cite{Soner1986}. 

Integro-differential operators have been extensively studied over the years by many authors. Regarding the qualitative properties of the solutions, the authors in \cite{Bass-Kassmann2005, Bass-Levin2002, Song-Vondavcek2004} established Harnack inequality by using probabilistic methods. It is important to mention that H\"older estimates do not follow from the Harnack inequality, since it requires the solution to be positive in the whole $\mathbb{R}^d$. Keeping this in mind, H\"older estimates were proved in \cite{Bass-Kassmann2005}, still using probabilistic techniques. In \cite{Silvestre2006} L. Silvestre provides an analytical proof of H\"older continuity, and also gives more flexible assumptions on the operator than the previous works. 

All these previous results enjoy the same feature: the estimates blow up as the order of the operator approaches 2. The first results that are uniform in the degree were given by L. Caffarelli and L. Silvestre in \cite{Caffarelli-Silvestre2009}. In that paper, the authors produced a series of results that extended the theory for second-order operators, such as a comparison principle, a nonlocal version of the Alexandroff-Bakelman-Pucci estimate, a Harnack inequality, and $C^{1,\alpha}$-estimates. Since the estimates are uniform as the degree of the operator goes to 2. These results can be seen as a natural extension of the regularity theory for elliptic PDEs. 

In \cite{Caffarelli-Silvestre2011} the authors extended their previous results to nontranslation-invariant equations using perturbative methods. In particular, they prove $C^{1,\alpha}$-regularity for solutions of equations that are close, in an appropriate sense, to another one with $C^{1,\bar{\alpha}}$-estimates. More or less simultaneously G. Barles, E. Chasseigne, and C. Imbert also dealt with nontranslation-invariant equations establishing H\"older regularity estimates. It is worth noticing that their assumptions are different from the ones in \cite{Caffarelli-Silvestre2011}, and the equations that they work with involve second and first-order terms, allowing some degenerate operators. For more details, see \cite{Barles-Chasseigne-Imbert2011}. 

A very interesting question that remains widely open is whether or not the solutions of 
\begin{equation}\label{eq_dosprazeres-topp}
	|Du|^p{\cal I}_{\sigma}(u,x) = f(x)\;\;\mbox{ in }\;\;B_1,  
\end{equation} 
are of class $C^{1, \alpha}$. In the local case, this type of problem drew the attention of many authors, see for instance \cite{Andrade-Pellegrino-Pimentel-Teixeira2022, Araujo-Ricarte-Teixeira2015, Birindelli-Demengel-2007, Birindelli-Demengel2014, Bronzi-Pimentel-Rampasso-Teixeira2020, Davila-Felmer-Quaas-2010, Imbert-Silvestre2013} just to cite a few. More generally, equations with double degeneracy of the form
\begin{equation}\label{eq_dd}
	(|Du|^p + a(x)|Du|^q)F(D^2u) = f(x)\;\;\mbox{ in }\;\;B_1,
\end{equation}
with $0 < p \leq q$ were also investigated. We refer the reader to \cite{DaSilva-Junior-Rampasso-Ricarte2021, DeFilippis2021,DeFillipis-Palatucci2019}, and the references therein. The main novelty in those works is that equation \eqref{eq_dd} is no longer homogeneous, which makes the scaling process more delicate. We stress that the operator in \eqref{eq_dd} is a nonvariational  counterpart of the extensively studied $(p,q)$-Laplacian, see for instance \cite{Colombo-Mingione-15,  DeFilippis-20, DeFilippis-21, DeFilippis-Mingione-2-21, DeFilippis-Mingione-21, DeFilippis-Palatucci-19, Mingione-Radulescu-21}.

In the nonlocal case, a partial result was given by D. Prazeres and E. Topp in \cite{Prazeres-Topp2021}. The authors suppose that $\sigma$ is sufficiently close to 2,  so that in the interactive process ${\mathcal I}_\sigma$ approaches $F$, where $F$ is a local operator. Under these assumptions, they prove that viscosity solutions to \eqref{eq_dosprazeres-topp}, with $p>0$, are of class $C^{1,\alpha}$. Their condition on $\sigma$ allows them to import some regulatity of the $F$-harmonic functions to the solutions of \eqref{eq_dosprazeres-topp}. In the general case, $\sigma \in (1, 2)$, there are not regularity results available for the solution of \eqref{eq_dosprazeres-topp}. One of our contribution in this manuscript is to show that the existence of a $C^{1, \alpha}_{loc}$ viscosity solution for a class of operators that includes the one in \eqref{eq_dosprazeres-topp}, for any $\sigma \in (1,2)$. This is the content of our first main result.

\begin{Theorem}\label{Degenerate Theorem at least one solution}
	Let $0<p\leq q$. Suppose A\ref{assump_degree}-A\ref{assump_regularf}, to be detailed further, hold true. Then, there exists  at least one $u\in C(\overline{B}_1)$ bounded viscosity solution to 
	\begin{equation*}
		- \left(|Du(x)|^p + a(x)|Du(x)|^q \right){\mathcal I}_{\sigma}(u, x)= f(x) \quad \text{in} \quad B_1,
	\end{equation*}
	such that $u \in C_{loc}^{1,\alpha}(B_1)$ with the estimate
	\[
	\|u \|_{C^{1,\alpha}(B_{1/2})} \leq C \left(\|u \|_{L^{\infty}(B_1)}+ \|u\|_{L^1_\sigma(\mathbb{R}^d)}+\| f\|^{\frac{\sigma - 1}{1+p}}_{L^{\infty}(B_1)} \right),
	\]
	where
	\[
	\alpha  \in \left(0, \min\left(\bar{\alpha}, \dfrac{\sigma - 1}{p+1} \right) \right),
	\]
	and $C = C(d, p, \lambda, \Lambda)$ is a positive constant. Here, $\bar\alpha$ is the exponent associated with the regularity of ${\mathcal I}_{\sigma}$ harmonic functions.
\end{Theorem}

We recall that a general regularity theory for viscosity solutions to (\ref{main_eq0}) when $0<p\leq q$ remains open, even in the case where $a\equiv0$. This happens because of the lack of an uniqueness result, which is the core of the arguments for the general regularity theory in local case. See also Remark \ref{rem_3}. 

In order to prove a general regularity result we need to restrict  ourselves to the case where $\sigma$ is close to 2, as in \cite{Prazeres-Topp2021}. In particular in extend the results in \cite{Prazeres-Topp2021} to a nonhomogeneous degeneracy setting. This is the content of our second result.

\begin{Theorem}\label{Degenerate Theorem}
	Let $u \in C(\overline{B}_1)$ be a bounded viscosity solution to
	\begin{equation*}
		- \left(|Du(x)|^p + a(x)|Du(x)|^q \right){\mathcal I}_{\sigma}(u, x)= f(x) \quad \text{in} \quad B_1,
	\end{equation*}
	with $0<p\leq q$. Suppose A\ref{assump_degree}-A\ref{assump_condition_on_kernels}, to be detailed further, hold true. Then, there exists $\sigma_0 \in (1, 2)$ sufficiently close to 2 such that if $\sigma_0 < \sigma < 2$, then $u \in C_{loc}^{1,\alpha}(B_1)$ with the estimate
	\[
	\|u \|_{C^{1,\alpha}(B_{1/2})} \leq C \left(\|u \|_{L^{\infty}(B_1)} + \|u\|_{L^1_\sigma(\mathbb{R}^d)}+\| f\|^{\frac{\sigma - 1}{1+p}}_{L^{\infty}(B_1)} \right),
	\]
	where
	\[
	\alpha  \in \left(0, \min\left(\bar{\alpha}, \dfrac{\sigma - 1}{p+1} \right) \right),
	\]
	and $C = C(d, p, \lambda, \Lambda)$ is a positive constant. Here, $\bar\alpha$ is the exponent associated with the regularity of $F$ harmonic functions.
\end{Theorem}

The constant $C$ that appears in the estimate above is uniform in $\sigma$ which means that it does not blow up as $\sigma$ approaches 2. Hence, we can see Theorem \ref{Degenerate Theorem} as an extension of \cite[Theorem 1]{DeFilippis2021}. 

In the singular case the difficulty starts with the notion of viscosity solution. Indeed, the definition considered in the local case \cite[Definition 2.2]{BD2006}, seems not to be suitable in our scenario, because whenever $x_0 \in \{Du=0\}$, we have that $u \equiv c$ in  $B_{r(x_0)}(x_0)$, which implies that $F(D^2u) = 0$ in $B_{r(x_0)}(x_0)$ and hence one could say that $0 \leq f(x_0)$ (in the case of supersolutions). On the other hand, in the nonlocal case, we still need to consider the quantity
\[
\int_{\mathbb{R}^{d}\setminus B_{r(x_0)}(x_0)}(u(x_0) - u(y))K(y)dy,
\]
which depends on $r(x_0)$ and may not vanish. To overcome this, we use the notion of {\it approximated viscosity solution}, which coincides with the usual notion of viscosity solution over the set $\{ Du \not=0 \}$, and are defined as the limit of the solutions of uniformly elliptic nonlocal equations . See Definition \ref{new_def} for the precise notion of approximated viscosity solution. Under this setting, we  establish $C^{1, \alpha}$-regularity for solutions of \eqref{main_eq0}. More precisely, we prove the following: 


\begin{Theorem}\label{Singular Theorem}
	Let $u \in C(\overline{B}_1)$ be a bounded approximated viscosity solution to
	\begin{equation*}
		- \left(|Du(x)|^p + a(x)|Du(x)|^q \right){\mathcal I}_{\sigma}(u, x)= f(x) \quad \text{in} \quad B_1,
	\end{equation*}
	with $-1<p\leq q<0$. Assume A\ref{assump_degree}-A\ref{assump_regularf}, to be detailed later, are in force. Then $u \in C_{loc}^{1,\alpha}(B_1)$ and we have the estimates
	\[
	\|u \|_{C^{1,\alpha}(B_{1/2})} \leq C \left(\|u \|_{L^{\infty}(B_1)} + \|u\|_{L^1_\sigma(\mathbb{R}^d)}+\| f\|^{\frac{\sigma - 1}{1+p}}_{L^{\infty}(B_1)} \right).
	\]
	for some  $\alpha \in (0, 1)$ and $C = C(d, p, \lambda, \Lambda)$ is a positive constant.
\end{Theorem}



The remainder of this article is structured as follows: in the second section, we collect some auxiliary results and present our assumptions. In the third section, we investigate the H\"older regularity for the gradient of the solutions in the degenerate case. The last section is devoted to the proof of the regularity estimates in H\"older spaces for the singular scenario.

\section{Preliminaries}\label{sec_not}

In this section, we gather basic notions and detail our main assumptions used throughout this paper. In what follows, we present the definition of the nonlocal operator that we work with in this article.

\begin{Definition}
Consider $\sigma \in(1, 2)$, and constants $0<\lambda\leq\Lambda$. We define the family  ${\mathcal K}_0$, as the set of  measurable kernels $K:{\mathbb R}^d \setminus \left\{ 0  \right \}\rightarrow {\mathbb R}$ satisfying
\begin{equation*}\label{kernel_condition}
\lambda\dfrac{C_{d,\sigma}}{|x|^{d+\sigma}}\leq K(x)\leq \Lambda\dfrac{C_{d,  \sigma}}{|x|^{d+\sigma}}, 
\end{equation*}
where $C_{d,\sigma}>0$ is a normalizing constant.
\end{Definition}

\begin{Definition}\label{nonlocal operator0}
Given $u:\mathbb{R}^d \rightarrow \mathbb{R}$, $\Omega \subseteq \mathbb{R}^d$ a measurable set and $K\in \mathcal{K}_0$, we define the operator $I_K$ as
\begin{equation}\label{nonlocal operator}
I_{K}[\Omega](u, x) := C_{d,\sigma}\mbox{P.V.} \displaystyle\int_{\Omega}\left( u(x+y) - u(x) \right)K(y)\; {\bf d}y,
\end{equation}
where $\mbox{P.V.}$ denotes the Cauchy principal value of the integral.
\end{Definition}
\begin{Definition}
We say that $u:\mathbb{R}^d \rightarrow \mathbb{R}$ belongs to $L^{1}_{\sigma}(\mathbb{R}^d)$, if

\begin{equation}\label{growth at infinity}
\|u\|_{L^1_\sigma}: =\displaystyle \int_{\mathbb{R}^d} |u(x)|\dfrac{1}{1 +|x|^{d+\sigma}} { \bf d}x < +\infty.
\end{equation}
\end{Definition}

Throughout this manuscript, we consider different ranges of $\sigma$ to obtain the main results. For this reason, we fix the notation $\mathcal{I}_{\sigma}(u, x)$ to emphasize the dependence on $\sigma$ for the class of nonlocal operators as defined in \eqref{nonlocal operator}. 

In the proofs of Theorem \ref{Degenerate Theorem at least one solution} and Theorem \ref{Degenerate Theorem}, we make use of some scaled functions that satisfy a variant of equation \eqref{main_eq0}, namely:
 
\begin{equation*}
- \left(|Du + \xi|^p + a(x)|Du + \xi|^q \right){\mathcal I}_{\sigma}(u, x)= f(x) \quad \text{in} \quad B_1,
\end{equation*}
where $\xi \in \mathbb{R}^d$. Hence, we define the solutions and we prove some results for the equation above instead  of only for equation \eqref{main_eq0}. We first present the definition of viscosity solutions to \eqref{main_eq0}. Let us consider a collection of kernels $\left \{K_{ij}\right\}_{i, j} \subseteq \mathcal{K}_0$.

\begin{Definition}[Viscosity solution]
We say that $u \in  C(\overline{B}_1)\cap L^1_\sigma$ is a viscosity subsolution to \eqref{main_eq0}, if for any $\varphi \in C^2(\mathbb{R}^d)$ and for all $x_0 \in B_1$ such that $u - \varphi$ has a local maximum at $x_0$,  we have 
\[
- \left(|D\varphi + \xi|^p + a(x_0)|D\varphi + \xi|^q \right){\mathcal I}_{\delta}(u, \varphi, x_0)\leq f(x_0)\quad \text{in} \quad B_1,
\]
where the operator $\mathcal{I}_{\delta}$ is given by
\begin{equation*}
\mathcal{I}_{\delta}(u, \varphi, x):=\inf_i\sup_j\left(I_{K_{ij}}[B_{\delta}](\varphi, x) - I_{K_{ij}}[B^c_{\delta}](u, x)\right).
\end{equation*}
Similarly, we say that $u\in C(\mathbb{R}^d)$ is a viscosity supersolution to \eqref{main_eq0}, if for any $\varphi \in C^2(\mathbb{R}^d)$ and for all $x_0 \in B_1$ such that $u - \varphi$ has a local minimum at $x_0$, we have 
\[
- \left(|D\varphi + \xi|^p + a(x_0)|D\varphi + \xi|^q \right){\mathcal I}_{\delta}(u, \varphi, x_0)\geq f(x_0) \quad \text{in} \quad B_1.
\]
Finally, we say that $u\in C(\mathbb{R}^d)$ is a viscosity solution to \eqref{main_eq0} if it is both a viscosity subsolution and supersolution.
\end{Definition}

In the next, we define the so-called approximated solutions. Its main purpose is to allow us to deal with the singular case,  but it is also a fundamental notion in the proof of Theorem \ref{Degenerate Theorem at least one solution}. 

\begin{Definition}[Approximated viscosity solution]\label{new_def}
Let $\xi $ be a vector in $ \mathbb{R}^d$. We say that $u \in C(\overline{B}_1)\cap L^{1}_{\sigma}(\mathbb{R}^d)$ is an approximated viscosity solution to 
\begin{equation}\label{eq_scaled}
-\left(|Du + \xi|^p + a(x)|Du + \xi|^q\right){\mathcal I}_{\sigma}(u, x) = f \;\;\mbox{ in }\;\;B_1,    
\end{equation}
if there are sequences $(u_j)_{j\in \mathbb{N}} \in C(B_1)\cap L^{1}_{\sigma}(\mathbb{R}^d)$, $(\xi_j)_{j\in \mathbb{N}} \in \mathbb{R}^d$, $(c_j)_{j\in \mathbb{N}} \in \mathbb{R}^+$ and $C_1 > 0$ satisfying 
\[
u_j \mbox{ converges locally uniformly to } u  \mbox{ in  } B_1,
\]
\[
\|u_j\|_{L^1_\sigma} \leq C_1(1 + |x|^{1+\alpha})\;,\;\; 1+\alpha \in(0,\sigma),
\]
\[
\xi_j \to \xi \;, \; c_j \to 0,
\]
such that $u_j$ is a viscosity solution of
\[
-\left((|Du_j + \xi_j|+c_j)^p + a(x)(|Du_j + \xi_j|+c_j)^q\right){\mathcal I}_\sigma(u_j,x) =  f \;\;\mbox{ in }\;\;B_1.
\]
Moreover, if $0< p \leq q$, $c_j$ has also to satisfy $jc_j^p \to \infty$, as $j \to \infty$. If $\xi = 0$, then we take $\xi_j = 0$ for all $j\in \mathbb{N}$.
\end{Definition}


\begin{Remark}\label{rem_2}
If $\varphi \in C^2(\mathbb{R}^d)$ is a test function touching $u$ from above at a point $x_0$ then, because $u_j \to u$ locally uniformly in $B_1$, we have that there exists a sequence $(\varphi_j)_{j\in \mathbb{N}}$ such that $\varphi_j \to \varphi$ locally uniformly in $B_1$ and $\varphi_j$ touches $u_j$ from above at a point $x_j$, where $x_j \to x_0$. Since $u_j$ is a viscosity solution, this implies, 
\[
-\left((|D\varphi_j(x_j)|+c_j)^p + a(x_j)(|D\varphi_j(x_j)|+c_j)^q\right){\mathcal I}_{\delta}(u_j, \varphi_j, x_j) \leq f(x_j).
\]
\end{Remark}

\begin{Remark}[Scaling properties]
Throughout the paper, we require
\begin{equation}\label{prop_scal}
\|f\|_{L^\infty(B_1)} \leq \varepsilon,
\end{equation} 
for some $\varepsilon$ to be determined. The condition in \eqref{prop_scal} is not restrictive. In fact, consider the function
\[
v(x) = \dfrac{u(x)}{K},
\]
with $K > 0$. Notice that $v$ is also a viscosity solution to \eqref{main_eq0} in $B_1$, with 
\[
\tilde{a}(x) := K^{q-p}a(x) \;\;\;\mbox{ and }\;\;  \tilde{f}(x) = \frac{1}{K^{p+1}}f(x).
\]  
Hence, by choosing 
\[
K = \left(\varepsilon^{-1}\|f\|_{L^\infty(B_1)}\right)^{\frac{1}{p+1}},
\]
we can assume \eqref{prop_scal} without loss of generality.
\end{Remark}

\subsection{Main assumptions}

In what follows, we detail the assumptions of the paper. The first one  concerns the degree of the operator ${\mathcal I}_{\sigma}$.

\begin{Assumption}[Degree of the operator]\label{assump_degree}
 We suppose that $1 < \sigma <   2$. 
\end{Assumption}

To guarantee that the nonlocal operator in \eqref{nonlocal operator} is well-defined, we must impose further conditions on the growth at infinity of the function $u$. This is the content of our next assumption.

\begin{Assumption}[Growth condition at infinity]\label{assump_growth_of_u}
We assume that $u \in L^{1}_{\sigma}(\mathbb{R}^d)$. 

\end{Assumption}

We remark that bounded functions satisfy the growth condition \eqref{growth at infinity}. The next assumption concerns the source term.

\begin{Assumption}[Regularity of the source term]\label{assump_regularf}
 We suppose that $f \in L^\infty(B_1)\cap C(B_1)$. 
\end{Assumption}

Although we require $f \in C(B_1)$, all the estimates will only depend on the $L^\infty(B_1)$ norm of $f$. The next assumption concerns the regulairy of $a$.
Our last assumption ensures the convergence of the operator ${\cal I}_\sigma$, as $\sigma$ goes to 2.

\begin{Assumption}\label{assump_condition_on_kernels}
Let $\left \{ K_{ij}\right\}_{i, j}$ be a collection of kernels in $\mathcal{K}_0$. There exists a modulus of continuity  $\omega$ and  $\left\{ k_{ij}\right \}_{i, j} \in (\lambda, \Lambda)$ satisfying the estimates
\begin{equation}
\left|K_{ij}(x)|x|^{d + \sigma} + k_{ij} \right| \leq \omega(|x|) \;\; \text{for\: all} \;i, j\;\; \text{and} \;\;|x|\leq 1.
\end{equation}
\end{Assumption}

\section{Analysis of the degenerate case} 

In this section, we assume $0 < p \leq q$, \emph{i. e.}, we are considering the degenerate scenario. We present the results in two subsections: The first one is dedicated to the proof of Theorem \ref{Degenerate Theorem at least one solution}, while in the second one we give a proof of Theorem \ref{Degenerate Theorem}.

\subsection{$C^{1, \alpha}$-regularity via approximated viscosity solutions}

This subsection is devoted to the proof of Theorem \ref{Degenerate Theorem at least one solution}. We start with some properties of approximated viscosity solutions to \eqref{main_eq0}. The first one states that the set of approximated viscosity solutions is contained in the set of viscosity solutions.

\begin{Proposition}\label{equiv_def}
Let $u \in C(\overline{B}_1)$ be an approximated viscosity solution to \eqref{eq_scaled}. Suppose that A\ref{assump_degree}-A\ref{assump_regularf} are in force. Then $u$ is a viscosity solution to the same equation.  
\end{Proposition}

\begin{proof}
Let $\varphi \in C^2(\mathbb{R}^d)$ and $x_0 \in B_1$ such that $u-\varphi$ has a local maximum at $x_0$. From Definition \ref{new_def} and Remark \ref{rem_2} we obtain 
\[
-\left[(|D\varphi_j(x_j) + \xi_j|+c_j)^p + a(x_j)(|D\varphi(x_j) + \xi_j|+c_j)^q\right]{\mathcal I}_{\delta}(u_j, \varphi_j, x_j) \leq f(x_j),
\]
where $\varphi_j \to \varphi$ and $u_j \to u$ locally uniformly in $B_1$, $x_j \to x_0$, $\xi_j \to \xi$ and $c_j \to 0$. By passing the limit as $j \to \infty$, we get that  
\[
-\left(|D\varphi(x_0) + \xi|^p + a(x_0)|D\varphi(x_0) + \xi|^q\right){\mathcal I}_{\delta}(u, \varphi, x_0) \leq f(x_0),
\]
which implies that $u$ is a viscosity subsolution to \eqref{eq_scaled}. We recall that the convergence ${\mathcal I}_{\delta}(u_j, \varphi_j, x_j) \to {\mathcal I}_{\delta}(u, \varphi, x_0)$ follows from \cite[Lemma 5]{Caffarelli-Silvestre2011}. Similarly, we can prove that $u$ is also a viscosity supersolution. This finishes the proof.
\end{proof}

\begin{Remark}\label{rem_3}
Due to the absence of a uniqueness result, the reverse statement may not be true. 
\end{Remark}

Next, we present a stability type result for approximated viscosity solutions of \eqref{eq_scaled}. 

\begin{Proposition}[Weak stability of approximated viscosity solutions]\label{prop_stability}
Let $(u_j)_{j \in \mathbb{N}} \in C(\overline{B}_1)$, $(\xi_j)_{j \in \mathbb{N}} \in \mathbb{R}^d$ and $f_j \in C(B_1)\cap L^\infty(B_1)$ be such that 
\[
-\left(|Du_j + \xi_j|^p + a_j(x)|Du_j+\xi_j|^q\right){\mathcal I}_\sigma(u_j,x) = f_j \;\;\mbox{ in }\;\;B_1
\]
in the approximated viscosity solution sense. Assume that A\ref{assump_degree}-A\ref{assump_regularf} hold true. Suppose further that there exists a function $u \in C(\overline{B}_1) \cap L^{1}_{\sigma}(\mathbb{R}^d)$ such that $u_j \to u$ locally uniformly in $B_1$ and $\|f_j\|_{L^\infty(B_1)} \leq 1/j$. Then, $u$ is a viscosity solution to
\begin{equation}\label{eq_stab}
{\mathcal I}_\sigma(u,x) = 0 \;\;\mbox{ in }\;\;B_1.
\end{equation}

\end{Proposition}

\begin{proof}
By Definition \ref{new_def}, there exist sequences $(u_j^m)_{m\in \mathbb{N}} \in C(B_1)\cap L^{1}_{\sigma}(\mathbb{R}^d)$, $(\xi_j^m)_{m\in \mathbb{N}} \in \mathbb{R}^d$ and $(c_m)_{j\in \mathbb{N}} \in \mathbb{R}^+$ satisfying 
\[
u_j^m \mbox{ converges locally uniformly to } u_j  \mbox{ in  } B_1,
\]
\[
c_m \to 0,\;\; mc_m^p \to \infty, 
\]
as $m \to \infty$, such that $u_j^m$ is a viscosity solution of
\[
-\left[(|Du_j^m + \xi_j^m|+c_m)^p + a_j(x)(|Du_j^m + \xi_j^m| + c_m)^q\right]{\mathcal I}_\sigma(u_j^m,x) = f_j \;\;\mbox{ in }\;\;B_1,
\]
for each fixed $j$. Now, we consider the sequence $(u_j^j)_{j\in \mathbb{N}}$. We have that $u_j^j \to u$ locally uniformly in $B_1$, and $u_j^j$ is a viscosity solution to
\[
-\left[(|Du_j^j + \xi_j^j|+c_j)^p + a_j(x)(|Du_j^j + \xi_j^j| + c_j)^q\right]{\mathcal I}_\sigma(u_j^j,x) = f_j \;\;\mbox{ in }\;\;B_1.
\]
Therefore, $u^j_j$ is also a viscosity solution to
\begin{align*}
-{\mathcal I}_\sigma(u_j^j,x) & = \dfrac{f_j}{\left[(|Du_j^j + \xi_j^j|+c_j)^p + a_j(x)(|Du_j^j + \xi_j^j| + c_j)^q\right]} \\
 & \leq \dfrac{\|f_j\|_{L^\infty(B_1)}}{\left[(|Du_j^j + \xi_j^j|+c_j)^p\right]} \\
 & \leq \dfrac{1}{jc_j^p}.
\end{align*}
Since by definition $jc^p_j \to \infty$, as $j \to \infty$, we infer that $u$ is a viscosity subsolution to
\[
-{\mathcal I}_\sigma(u,x) \leq 0 \;\;\mbox{ in }\;\;B_1.
\]
Similarly, we show that $u$ is also a viscosity supersolution. This finishes the proof.
\end{proof}

Now, for each $j \in \mathbb{N}$, we consider the nonlocal uniformly elliptic equation
\begin{equation}\label{eq_reg2}
-\left[(|Du_j|+c_j)^p + a(x)(|Du_j|+c_j)^q\right]{\cal I}_{\sigma}(u_j,x) = f(x) \;\;\mbox{ in }\;\;B_1.
\end{equation}
We are going to show that a sequence $(u_j)_{j\in\mathbb{N}} \in C(B_1)\cap L^1_\sigma(\mathbb{R}^d)$ of viscosity solutions to \eqref{eq_reg2} converges to a function $u \in C(B_1)\cap L^1_\sigma(\mathbb{R}^d)$, where $u$ is an approximated viscosity solution to \eqref{main_eq0} with a suitable boundary data. In order to do that, we stablish a compactness result for the solutions of \eqref{eq_reg2}.

\begin{Lemma}\label{Lipschitz_Regularity2}
Let  $u \in C(\overline{B}_1)$ be a viscosity solution to \eqref{eq_reg2}. Suppose that A\ref{assump_degree}-A\ref{assump_regularf} are in force. Then $u_j$ is locally Lipschitz continuous,\emph{ i. e.,}
\begin{equation*}
|u_j(x) - u_j(y)| \leq C|x - y| \quad  \text{for all}\: \: x, y  \in B_{1/2},
\end{equation*}
where the constant $C>0$ does not depend on $j$. 
\end{Lemma}

\begin{proof} 
We consider $\tilde{\psi}: \mathbb{R}^d \rightarrow \mathbb{R}$ a nonnegative and smooth function such that 
\[
\tilde{\psi}=0 \quad \text{in} \quad  B_{1/2}, \quad \tilde{\psi} = 1 \quad \text{in} \quad B^c_{3/4}, \quad \]
and we define
\[
\psi = (\osc_{B_1} u_j + 1)\tilde{\psi} .
\]
Let $\alpha \in (0, 1)$ to be determined later, and $\varphi: [0, +\infty) \rightarrow \mathbb{R}$ be defined as 
\[
\varphi(t)=
\left\{
\begin{array}{cc}
t - t^{1+\alpha}, & t \in [0, t_0]\\
\varphi(t_0), & t> t_0,\\
\end{array}
\right.
\]
where $t_0>0$ is chosen sufficiently small in order to have $\varphi >0$ in $(0, +\infty)$.  It follows from the definition of $\varphi$ that it is smooth in $(0, t_0)$ and differentiable in $[0, t_0)$.

Consider $\phi, \Phi:\mathbb{R}^d \times \mathbb{R}^d \rightarrow \mathbb{R}$ defined by
$\phi, \Phi:\mathbb{R}^d \times \mathbb{R}^d \rightarrow \mathbb{R}$ as
\[
\phi(x, y):= L\varphi(|x- y|) + \psi(y)
\]
and 
\[
\Phi(x, y) := u_j(x) - u_j(y) - \phi(x, y).
\]
From the continuity property of $\Phi$, we have $\Phi$ attains its maximum in $\overline{B}_1 \times \overline{B}_1$ at $(\bar{x}, \bar{y})$. Hence, it is enough to prove that $\Phi(\bar{x}, \bar{y}) \leq 0$. Suppose by contradction that 
\[
\Phi(\bar{x}, \bar{y}) > 0.
\]
It follows that
\[
u_j(\bar{x}) - u_j(\bar{y}) - L\varphi(|\bar{x}- \bar{y}|) - \psi(\bar{y})>0.
\]
Therefore 
\[
\osc_{B_1} u - (\osc_{B_1} u +1 )\tilde{\psi}(\bar{y}) - L \varphi(|\bar{x} - \bar{y}|) >0.
\]
Observe that if $\bar{y} \in B^c_{3/4}$, then $\tilde{\psi} \equiv 1$ which implies 
\[
-1 - L\varphi(|\bar{x} - \bar{y}|) >0,
\]
which is a contradiction, hence $\bar{y} \in B_{3/4}$. Since $L\varphi(|\bar{x} - \bar{y}|) \leq \osc_{B_1} u$, by taking $L$ sufficiently large, we may assume that $\bar{x} \in B_{7/8}$. Finally, it is straightforward to notice that $\bar{x}\not= \bar{y}$, otherwise we would have $\Phi (\bar{x}, \bar{y})<0$.

Now, we compute $D_x\phi$ and $D_y\phi$ at $(\bar{x}, \bar{y})$
\[
D_x\phi(\bar{x}, \bar{y}) =  L_2{\varphi}^{\prime}(|\bar{x} - \bar{y}|)|\bar{x} - \bar{y}|^{-1}(\bar{x} - \bar{y})
\]
and 
\[
- D_y\phi(\bar{x}, \bar{y}) = L_2{\varphi}^{\prime}(|\bar{x} - \bar{y}|)|\bar{x} - \bar{y}|^{-1}(\bar{x} - \bar{y}) - D\psi(\bar{y}).
\]
Let us denote
\[
\xi_{\bar{x}} := D_x\phi(\bar{x}, \bar{y}) \quad \text{and} \quad
\xi_{\bar{y}} := D_y\phi(\bar{x}, \bar{y}).
\]
We have that $u_j - \Phi_{\bar{y}}$ attains its maximum at $\bar{x}$, where ${\Phi}_{\bar{y}}(x) := u_j(\bar{y}) + \phi(x, \bar{y})$. Hence, since $u_j$ is a viscosity solution to \eqref{eq_reg2}, we obtain the following viscosity inequality
\[
- \left[(|{\xi}_{\bar{x}}| + c_j)^p + a(\bar{x})(|{\xi}_{\bar{x}}|+c_j)^q \right]{\mathcal I}_{\delta}(u_j, {\Phi}_{\bar{y}}, \bar{x})\leq f(\bar{x}).
\]
Now, we take $L$ sufficiently large, only depending on $\mbox{osc}_{B_1}u$, so that
\[
(|{\xi}_{\bar{x}}| + c_j), (|{\xi}_{\bar{y}}| + c_j) \geq 1/2, 
\]
which implies
\[
\begin{array}{ccl}
- {\mathcal I}_{\delta}(u, {\Phi}_{\bar{y}}, \bar{x})
&\leq & \dfrac{ \|f\|_{L^{\infty}(B_1)}}{\left[(|{\xi}_{\bar{x}}| + c_j)^p + a(\bar{x})(|{\xi}_{\bar{x}}|+c_j)^q \right]} \vspace{0.3cm} \\
& \leq & \dfrac{ \|f\|_{L^{\infty}(B_1)}}{(|{\xi}_{\bar{x}}| + c_j)^p} \vspace{0.3cm} \\
&\leq& 2\|f\|_{L^{\infty}(B_1)}. \\
\end{array}
\]
Similarly we can obtain $- \mathcal{I}_{\delta}(u, - \Phi_{\bar{x}}, \bar{y}) \geq -2\|f\|_{L^{\infty}(B_1)}$, and consequently
\[
{\mathcal I}_{\delta}({u, \Phi}_{\bar{y}}, \bar{x}) - \mathcal{I}_{\delta}(u, - \Phi_{\bar{x}}, \bar{y})\geq -  4\|f\|_{L^{\infty}(B_1)}. 
\]
At this point, we concentrate in estimating the left hand side of the inequality above, that consists of uniformly elliptic nonlocal operators. The next step follows the arguments in \cite{Barles-Chasseigne-Ciomaga-Imbert-12, Barles-Chasseigne-Imbert-11} and we present here for the sack of completeness.
	
	
For $|z| \leq 1/10$, we have $\bar x + z, \bar y + z \in \overline B_1$. Then, there exists a kernel $K$ in the family ${\mathcal K}_0$ such that
	\begin{equation}\label{Roma}
	-C(\|f\|_\infty + \osc_{B_1}u_j + \|u_j\|_{L^1_\sigma}) - 1 \leq I_1 + I_2, 
	\end{equation}
	where 
	\begin{align*}
	I_1 & := I_K[B_{\delta}](\phi(\cdot, \bar y), \bar x) - I_K[B_{\delta}](-\phi(\bar x, \cdot), \bar y), \\
	I_2 & := I_K[B_{1/10} \setminus B_\delta](u_j, \xi_{\bar x}, \bar x) - I_K[B_{1/10} \setminus B_\delta](u_j, \xi_{\bar y}, \bar y),
	\end{align*}
	and the constant $C > 0$ depends on the ellipticity constants.

 We denote $e = \bar x - \bar y$, $\hat e = e/|e|$ and 
	\begin{align*}\label{C}
	\mathcal{C} = \{ z \in B_\rho : |\langle \hat e, z \rangle| \geq (1 - \eta) |z|\},
	\end{align*}
	for constants $\eta \in (0,1)$ and $\rho \in (0,1/10)$ to be fixed.  Since $(\bar x, \bar y)$ is a maximum point of $\Phi$, for all $z \in \overline B_{1/10}$ we have the following inequalities
	
\begin{eqnarray*}\label{ineq}
 u_j(\bar x + z) - u_j(\bar x) &\le& u_j(\bar y + z)  - u_j(\bar y) 
+ \phi(\bar x + z, \bar y) - \phi(\bar x, \bar y), \nonumber\\
 u_j(\bar y + z) - u_j(\bar y) &\geq &  -\phi(\bar x, \bar y+z)+\phi(\bar x,\bar y), \\
 u_j(\bar x + z) - u_j(\bar x) &\leq & \phi(\bar x +z,\bar y)- \phi(\bar x,\bar y), \nonumber
\end{eqnarray*}
and hence, it is possible to get that
	\begin{align*}
	I_K[\mathcal{C} \setminus B_\delta](u_j, \xi_{\bar x}, \bar x) & \leq I_K[\mathcal{C}\setminus B_\delta](\varphi, e), \\
	I_K[\mathcal{C} \setminus B_\delta](u_j, \xi_{\bar y}, \bar y) & \geq -I_K[\mathcal{C}\setminus B_\delta](\varphi, e) - C\osc_{B_1}u_j,
	\end{align*}
	where the constant $C > 0$ is uniformly bounded from above and from below as $\delta \to 0$. Moreover, for every set $\mathcal O \subset \mathbb{R}^d$ such that $\bar x + z, \bar y + z \in \overline B_1$, we have
	\begin{equation*}
	I_K[\mathcal O](u_j,  \varphi, \bar{x}) - I_K[\mathcal O](u_j, \varphi, \bar{y}) \leq I_K[\mathcal O](\psi, \bar{y}) \leq C \Lambda \osc_{B_1}u_j,
	\end{equation*}
where the latter inequality follows from the smoothness of $\psi$.
	
	Then, we plug these inequalities into \eqref{Roma} to conclude that
	\begin{equation*}
	-C(\|f\|_\infty + \osc_{B_1}u_j + \|u_j\|_{L^1_\sigma} + 1) \leq 2 I_K[\mathcal{C} \setminus B_\delta](\varphi, e) + I_1,
	\end{equation*}
and we notice that the term $I_1 \to 0$ as $\delta \to 0$, meanwhile, by Dominated Convergence Theorem, we have $I_K[
\mathcal{C} \setminus B_\delta](\varphi, e) \to I_K[\mathcal{C}](\varphi, e)$ as $\delta \to 0$.

	We make $\delta \to 0$, and take $\eta = c_1 |e|^{2\alpha}, \rho = c_1 |e|^\alpha$ for some constant $c_1 > 0$ universal. Using the
estimates  in \cite[Corollary 9]{Barles-Chasseigne-Ciomaga-Imbert-12}, we have that
	\begin{equation}\label{crucial}
	I_K[\mathcal{C}](\varphi, e) \leq - c L |e|^{1 - \sigma + \alpha (N + 2 - \sigma)},
	\end{equation}
	for some $c=c(\lambda, \Lambda, d)> 0$ such that
 $\epsilon_0 \leq c \leq \epsilon_0^{-1}$ for all $\sigma \in (1,2)$, where $\epsilon_0 \in (0,1)$ depends only on $d$.

Now, we fix $\alpha > 0$ small enough to get 
	$$
	-\theta := 1 - \sigma + \alpha(N + 2 - \sigma) \leq (1 - \sigma)/2 < 0.
	$$
Thus, replacing the estimates above into~\eqref{Roma} we arrive at
	\begin{equation*}
	-C(\|f\|_\infty + \osc_{B_1} (u_j) + \|u_j\|_{L^1_\sigma} + 1) \leq -c L |\bar{x}-\bar{y}|^{-\theta}, 
	\end{equation*}
and since 
\[
|\bar{x}-\bar{y}|\leq C_2 \mbox{osc}_{B_1}u_j L^{-1}
\]
we obtain
\[
-C_1\left( \|f\|_{L^\infty(B_1)} + \mbox{osc}_{B_1}u_j + \|u_j\|_{L^1_\sigma} +1 \right) \leq -cC_2L^{\theta}.
\]
Hence by taking 
\[
L := C_3\left( \|f\|_{L^\infty(B_1)} + \mbox{osc}_{B_1}u_j +1 \right), 
\]
where $C_3$ is a large enough constant, depending only on $\lambda$, $\Lambda$, $d$ we get a contradiction. This finishes the proof.
\end{proof}

\begin{Proposition}[Existence of approximated viscosity solutions]\label{prop_exis2}
Suppose that A\ref{assump_degree}-A\ref{assump_regularf} hold true. Then, there exists at least one approximated viscosity solution $u\in C(B_1)$ to \eqref{main_eq0}.
\end{Proposition}

\begin{proof}
Consider the equation
\[
\left\{
\begin{array}{rcl}
-\left[(|Du_j|+\frac{1}{j^{1/2p}})^p + a(x)(|Du_j|+\frac{1}{j^{1/2p}})^q\right]{\mathcal I}_{\sigma}(u_j,x) & = & f(x) \;\;\mbox{ in }\;\;B_1 \\
u_j(x) & = & g(x)\;\;\mbox{ in }\;\;\mathbb{R}^d\setminus B_1,
\end{array}
\right.
\]
where $g \in L^1_{\sigma}(\mathbb{R}^d)$  satisfies 
\[
g(x) \leq 1 + |x|^{1+\alpha}\;\;\mbox{ for all }\;\; x\in \mathbb{R}^d,
\]
for any $1+\alpha \in (0,\sigma)$. Since the operator is nonlocal uniformly elliptic, the existence of a viscosity solution $u_j$ is assured by \cite{Barles-Chasseigne-Imbert-2008}. By Lemma \ref{Lipschitz_Regularity2}, we have that the solution $u_j \in C^{0,1}_{loc}(B_1)$, with estimates that are independent of $j$. Hence, there exists a function $u_{\infty} \in C^{\alpha}_{loc}(B_1)$, for some $\alpha \in (0,1)$, such that $u_j \to u_{\infty}$ locally uniformly in $B_1$, and $u_{\infty} \equiv g$ outside $B_1$. By taking $c_j := j^{1/2p}$, we have directly by definition that $u_\infty$ is an approximated viscosity solution to \eqref{main_eq0}.
\end{proof}

As we mentioned before, we will need to deal with equations of the form \eqref{eq_scaled}, hence we first prove some level of compactness for the solutions of such equations. 

\begin{Lemma}\label{Lipschitz_Regularity}
Let  $u \in C(\overline{B}_1)$ be a viscosity solution to \eqref{eq_scaled}. Suppose that A\ref{assump_degree}-A\ref{assump_regularf} are in force. There exists a constant $c_0 >1$ such that if $|\xi|\geq c_0$, then u is locally Lipschitz continuous,\emph{ i. e.,}
\begin{equation*}
|u(x) - u(y)| \leq C|x - y| \quad  \text{for all}\: \: x, y  \in B_{1/2},
\end{equation*}
where the constant $C>0$ does not depend on $\xi$. 
\end{Lemma}

\begin{proof} 
First, we rewrite equation \eqref{eq_scaled} as follows
\begin{equation}\label{eq_scal_norm}
-(|\hat{\xi} + b Du|^p + \hat{a}(x)|\hat{\xi} +b Du|^q )\mathcal{I}_{\sigma}(u, x) = b^pf(x) \quad \text{in} \quad B_1,
\end{equation}
where $\hat{\xi} = {\xi}/{|\xi|}$, $b = |\xi|^{-1}$ and $\hat{a}(x) = a(x)|\xi|^{q-p}$.

Considering the same notation as in Lemma \ref{Lipschitz_Regularity2}, we argue through a contradiction argument, by supposing 
\[
\Phi(\bar{x}, \bar{y}) > 0,
\]
where $\bar{x}, \bar{y} \in B_{9/10}$.

After this, we continue by computing $D_x\phi$ and $D_y\phi$ at $(\bar{x}, \bar{y})$, and we observe that $u - {\Phi}_{\bar{y}}$ attains a maximum at $\bar{x}$. Hence, since $u$ is a viscosity solution to \eqref{eq_scal_norm},
we obtain 
\[
- \left(|\hat{\xi} + b{\xi}_{\bar{x}}|^p + \hat{a}(\bar{x})|\hat{\xi} + b {\xi}_{\bar{x}}|^q \right){\mathcal I}_{\delta}(u, {\Phi}_{\bar{y}}, \bar{x})\leq b^p f(\bar{x}). 
\]
Similarly, we have the viscosity inequalities for $u$ at $\bar{y}$,
\[
- \left(|\hat{\xi} + b{\xi}_{\bar{y}}|^p + \hat{a}(\bar{y})|\hat{\xi} + b {\xi}_{\bar{y}}|^q \right){\mathcal I}_{\delta}(u, - {\Phi}_{\bar{x}}, \bar{x})\geq b^p f(\bar{y}).
\]

In the sequel, by choosing $c_0$ such that $b \leq 1/(2L(2 +\alpha))$, we may conclude that
\[
|\hat{\xi} + b{\xi}_{\bar{x}}|, |\hat{\xi} + b{\xi}_{\bar{y}}|\geq 1/2.
\]
Therefore, the former inequalities imply
\[
\begin{array}{ccl}
- {\mathcal I}_{\delta}(u, {\Phi}_{\bar{y}}, \bar{x})
&\leq & \dfrac{ b^p \|f\|_{L^{\infty}(B_1)}}{|\hat{\xi} + b{\xi}_{\bar{x}}|^p + \hat{a}(\bar{x})|\hat{\xi} + b {\xi}_{\bar{x}}|^q } \vspace{0.3cm} \\
&\leq&\dfrac{ b^p \|f\|_{L^{\infty}(B_1)}}{|\hat{\xi} + b{\xi}_{\bar{x}}|^p} \vspace{0.3cm} \\
&\leq& \dfrac{ 2^p \|f\|_{L^{\infty}(B_1)}}{{c^p_0}} \vspace{0.3cm} \\
&\leq& C_1\|f\|_{L^{\infty}(B_1)}. \\
\end{array}
\]
Similarly $- \mathcal{I}_{\delta}(u, - \Phi_{\bar{x}}, \bar{y}) \geq -C_1\|f\|_{L^{\infty}(B_1)}$ and consequently
\[
{\mathcal I}_{\delta}({u, \Phi}_{\bar{y}}, \bar{x}) - \mathcal{I}_{\delta}(u, - \Phi_{\bar{x}}, \bar{y})\geq -  2C_1\|f\|_{L^{\infty}(B_1)}. 
\]
From here we can proceed as in Lemma \ref{Lipschitz_Regularity2} to finish the proof.
\end{proof}

The next lemma deals with the other alternative of the norm of $\xi$, i.e., the case where $|\xi| \leq c_0$.

\begin{Lemma}\label{Lip_reg_2}
Let  $u \in C(\overline{B}_1)$ be a viscosity solution to \eqref{eq_scaled}. Suppose that A\ref{assump_degree}-A\ref{assump_regularf} hold true. If $|\xi|\leq c_0$, where $c_0$ is as in the previous lemma, then u is locally Lipschitz continuous,\emph{ i. e.,}
\begin{equation*}
|u(x) - u(y)| \leq C|x - y| \quad  \text{for all}\: \: x, y  \in B_{1/2},
\end{equation*}
where the constant $C>0$ does not depend on $\xi$. 
\end{Lemma}

\begin{proof}
The proof follows the same general lines as in the proof of Lemma \ref{Lipschitz_Regularity}. We define $\psi, \phi$ and $\Phi$ as in the previous lemma. As before, we can assume that $\Phi$ attains its maximum at $(\bar{x},\bar{y}) \in B_{7/8}\times B_{7/8}$  and we argue by contradiction assuming that $\Phi(\bar{x},\bar{y}) < 0$. By taking $L$ sufficient large, depending only on $\osc_{B_1} u$ and $c_0$ we can infer that
\[
|\xi + {\xi}_{\bar{x}}|, |\xi + {\xi}_{\bar{y}}| \geq 1.
\]
Since $u$ is a viscosity solution to \eqref{eq_scaled}, we have 
\[
- \left(|\xi + {\xi}_{\bar{x}}|^p + a(\bar{x})|{\xi} +{\xi}_{\bar{x}}|^q \right){\mathcal I}_{\delta}(u, {\Phi}_{\bar{y}}, \bar{x})\leq f(\bar{x}), 
\]
which implies 
\[
-{\mathcal I}_{\delta}(u, {\Phi}_{\bar{y}}, \bar{x}) \leq \|f\|_{L^\infty(B_1)}.
\]
Similarly $- \mathcal{I}_{\delta}(u, - \Phi_{\bar{x}}, \bar{y}) \geq -\|f\|_{L^{\infty}(B_1)}$ which implies
\[
{\mathcal I}_{\delta}({u, \Phi}_{\bar{y}}, \bar{x}) - \mathcal{I}_{\delta}(u, - \Phi_{\bar{x}}, \bar{y})\geq -  2\|f\|_{L^{\infty}(B_1)}. 
\]
From here we can proceed as in Lemma \ref{Lipschitz_Regularity2} to finish the proof.
\end{proof}

By combining Lemma \ref{Lipschitz_Regularity} and Lemma \ref{Lip_reg_2}, we obtain a compactness result for solutions of \eqref{eq_scaled}. In particular, we obtain Lipschitz estimates for the viscosity solutions of \eqref{main_eq0}.

\begin{Proposition}\label{Lip_reg}
Let  $u \in C(\overline{B}_1)$ be a viscosity solution to \eqref{eq_scaled}. Suppose that A\ref{assump_degree}-A\ref{assump_regularf} are in force. Then u is locally Lipschitz continuous,\emph{ i. e.,}
\begin{equation*}
|u(x) - u(y)| \leq C|x - y| \quad  \text{for all}\: \: x, y  \in B_{1/2},
\end{equation*}
where the constant $C>0$ does not depend on $\xi$. 
\end{Proposition}

The next lemma is the core of the arguments needed for the proof of Theorem \ref{Degenerate Theorem at least one solution}.

\begin{Lemma}\label{approx_lemma2}
Let $u \in C(\overline{B}_1)$ be a normalized approximated viscosity solution to \eqref{eq_scaled}. Suppose that A\ref{assump_degree}-A\ref{assump_regularf} are in force. Given $\mathfrak{M}, \delta>0$ and $\alpha \in (0, 1)$, there exists $\varepsilon>0$ such that if 
\begin{equation}\label{eq0_lemma2}
|u(x)|\leq \mathfrak{M}\left(1 + |x|^{1 +\alpha}\right) \quad \text{for\;all}\; \; x \in \mathbb{R}^d
\end{equation}
and 
\[
\| f\|_{L^{\infty}(B_1)} \leq \varepsilon,
\]
then we can find a function $h \in {\mathcal C}_{loc}^{1, \bar{\alpha}}(B_1)$ with $\bar{\alpha} \in (0, 1)$ satisfying 
\[
\|u - h\|_{L^{\infty}(B_{3/4})} \leq \delta.
\]
\end{Lemma}
\begin{proof}
We argue by contradiction. That is, there are $\mathfrak{M}_0, \delta_0 >0, \alpha_0 \in (0, 1)$ and sequences $(\xi_j)_{j \in \mathbb{N}}$,  $(u_j)_{j \in \mathbb{N}}$,  $(f_j)_{j \in \mathbb{N}}$ such that $u_j$ is an approximated viscosity solution to
\begin{equation}\label{eq_approx2}
-\left(|Du_j + \xi_j|^p + a_j(x)|Du_j + \xi_j|^q\right){\mathcal I}_{\sigma}(u_j, x)= f_j(x) \quad \text{in} \quad B_1,
\end{equation}
\[
|u_j(x)|\leq \mathfrak{M}_0\left(1 + |x|^{1 +\alpha_0}\right) \;\; \mbox{ in }\; \mathbb{R}^d,
\]
and 
\[
\| f_j\|_{L^{\infty}(B_1)} \leq 1/j,
\]
however
\begin{equation}\label{contradiction_argument2}
\|u_j - h\|_{L^{\infty}(B_{3/4})} \geq \delta_0
\end{equation}
for all $h \in {\mathcal C}_{loc}^{1, \bar{\alpha}}(B_1)$.

~Thanks to Proposition \ref{Lip_reg}, we can guarantee the existence of a function $u_\infty$ such that $u_j \rightarrow u_{\infty}$ locally uniformly in $B_{1}$, recall that approximated viscosity solutions are also viscosity solutions by Proposition \ref{equiv_def}. The contradiction assumptions combined with Proposition \ref{prop_stability} allow us to pass the limit in \eqref{eq_approx2} and conclude that $u_\infty$ is a viscosity solution to 
\[
-{\mathcal I}_{\sigma}(u_\infty, x) = 0 \quad \text{in} \quad B_{9/10}.
\]
Hence, we infer that $u_\infty$ is of class $C^{1, \alpha}_{loc}$. Finally, by taking $h = u_\infty$, we arrive in a contradition with \eqref{contradiction_argument2} for $j$ sufficiently large.
\end{proof}

\begin{Proposition}\label{Deg_Step1}
Let $u\in C(\overline{B}_1)$ be an approximated viscosity solution to \eqref{main_eq0}. Assume that A\ref{assump_degree}-A\ref{assump_regularf} hold true. Given $\mathfrak{M}>0$, there exists $\varepsilon>0$ such that, if 
\begin{equation*}
|u(x)|\leq \mathfrak{M}\left(1 + |x|^{1 +\alpha}\right) \quad \text{for\;all}\; \; x \in \mathbb{R}^d
\end{equation*}
and 
\[
\|f\|_{L^{\infty}(B_1)}\leq \varepsilon,
\]
one can find a constant $0 <\rho\ll 1/2$ and an affine function $\ell$ for which
\[
\|u - \ell \|_{L^{\infty}(B_{\rho})}\leq\rho^{1+\alpha},
\]
for every $\alpha \in (0, \bar{\alpha})$.
\end{Proposition}
\begin{proof}
Let $h \in \mathcal{C}^{1, \bar{\alpha}}(B_{3/4})$ be the function from the Lemma \ref{approx_lemma2} such that
\[
\|u - h\|_{L^{\infty}(B_{3/4})} \leq \delta.
\]
From the regularity available for $h$, we get
\[
\|h - [h(0)+ Dh(0)\cdot x] \|_{L^{\infty}(B_{\rho})} \leq C_0\rho^{1+\bar{\alpha}},
\]
where $C_0=C_0(\lambda, \Lambda, d, p)$ is a positive constant. Set $\ell(x):= h(0) + Dh(0)\cdot x$, by using the triangular inequality, we obtain
\[
\begin{array}{ccl}
\sup_{B_{\rho}}|u(x)-\ell(x)| &\leq & \sup_{B_{\rho}}|u(x)-h(x)| + \sup_{B_{\rho}}|h(x)-\ell(x)|\vspace{0.2cm}\\
&\leq& \delta + C_0\rho^{1+\bar{\alpha}}\\
&\leq& \rho^{1+\alpha},
\end{array}
\]
provided we choose $\rho$ and $\delta$ such that 
\begin{equation}\label{rho_choice}
\rho:= \min\left[ \left( \dfrac{1}{2C_0}\right)^{\frac{1}{\bar{\alpha}-\alpha}}, \left(\dfrac{1}{(1+C_0)100}\right)^{\frac{1}{\bar{\alpha}-\alpha}}\right]
\quad\text{and} \quad \delta:= \dfrac{\rho^{1+\alpha}}{2}.
\end{equation}
Notice that the universal choice of $\delta$ determines the value of $\varepsilon$ through the Lemma \ref{approx_lemma2}.
\end{proof}

\begin{Proposition}\label{Deg_step2}
Let $u \in C(\overline{B}_1)$ be an approximated viscosity solution to \eqref{main_eq0}. Suppose that A\ref{assump_degree}-A\ref{assump_regularf} are in force. Given $\mathfrak{M}>0$, there exists $\varepsilon>0$ such that, if 
\begin{equation*}
|u(x)|\leq \mathfrak{M}\left(1 + |x|^{1 +\alpha}\right) \quad \text{for\;all}\; \; x \in \mathbb{R}^d
\end{equation*}
and
\[
\| f\|_{L^{\infty}(B_1)}\leq \varepsilon,
\]
then we can find a sequence of affine functions $(\ell_k)_{k \in \mathbb{N}}$ of the form $\ell_k(x):= a_k + b_k \cdot x$ satisfying 
\[
\sup_{B_{\rho^k}}|u(x) - \ell_k(x)|\leq \rho^{k(1+\alpha)}
\]
such that
\begin{equation}\label{eq_06}
|a_{k+1}- a_{k}|\leq C \rho^{k(1+\alpha)} \quad \text{and}\quad |b_{k+1} - b_{k}|\leq  C \rho^{k \alpha}
\end{equation}
for all 
\begin{equation}\label{optimal_alpha}
\alpha  \in \left(0, \min\left(\bar{\alpha}, \dfrac{\sigma - 1}{p+1} \right) \right),
\end{equation}
where  $C=C(\lambda, \Lambda, d, p)$ is a positive constant.
\end{Proposition}
\begin{proof}
We proceed by induction argument on $j$. The case $j = 1$ is the content of Proposition \ref{Deg_Step1}. Suppose the statement have been verified for $j= 1, \ldots, k$. We shall prove the case $j = k+1$. Define the auxiliary function $v_k: \mathbb{R}^d \rightarrow \mathbb{R}$ such that
\[
v_k(x)= \dfrac{u(\rho^k x) - \ell_k(\rho^k x)}{\rho^{k(1+\alpha)}}.
\]
By the induction hypotheses $\| v_k\|_{L^{\infty}(\overline{B}_1)}\leq 1$. In addition, $v_k$ solves 
\[
- \left(|Dv_k +\rho^kb_k|^p + \tilde{a}(x)|Dv_k +\rho^kb_k|^q\right)\mathcal{I}_{\sigma}(x,v_k) = \tilde{f}(x),
\]
where $\tilde{a}(x)= a(\rho^kx)\rho^{k\alpha(q-p)}$ and $\tilde{f}(x)= \rho^{k\sigma - k(\alpha +\alpha p +1)}f(\rho^k x)$. Observe that $\tilde{a}(x)\geq 0$ and $\tilde{a} \in C(B_1)$. Let $\alpha$ be as in \eqref{optimal_alpha}, so that $\sigma - \alpha(1+p) - 1\geq 0$. This implies that $\| \tilde{f}\|_{L^{\infty}(B_1)}\leq \varepsilon$.
Now, we claim that
\begin{equation}\label{nonlocal_assumption_v_k}
|v_k(x)| \leq 1 + |x|^{1 + \bar{\alpha}} \quad \text{for\;all}\; \; x \in \mathbb{R}^d.
\end{equation}
Once we have verified \eqref{nonlocal_assumption_v_k}, we can apply the Proposition \ref{Deg_Step1} to $v_k$ and obtain
\[
\sup_{B_{\rho}}|v_k(x) - \tilde{\ell}_k(x)|\leq \rho^{1 +\alpha},
\]
consequently,
\[
\sup_{B_{\rho^{k+1}}}|u(x) - \ell_{k+1}(x)|\leq \rho^{k(1 +\alpha)},
\]
where $\ell_{k+1}(x):= \ell_k(x) + \rho^{k(1+\alpha)}\tilde{\ell_k}(\rho^{-k}x)$ with  $|a_{k+1} - a_{k}|\leq C \rho^{k(1+\alpha)}$ and $|b_{k+1} - b_{k}|\leq  C \rho^{k \alpha }$, finishing the proof.

We again resort to an induction argument. The case $k =0$, we take $ v_0 = u$. Suppose that we have proved the statement in the case $k=0,...,m$. We shall verify the case $k= m+1$. Notice that
\[
v_{m+1}(x)= \dfrac{v_m(\rho x) - \tilde{\ell}_m(\rho x)}{\rho^{1+\alpha}}.
\] 
If $|x| \rho>1/2$, we have
\[
\begin{array}{ccl}
|v_{m+1}(x)| &\leq& \rho^{-(1+\alpha)} \left(|v_m(\rho x)| +|\tilde{\ell}_m(\rho x)| \right)\vspace{0.2cm}\\
&\leq&\rho^{-(1+\alpha)}(1+ \rho^{1+\bar{\alpha}}|x|^{1+\bar{\alpha}}) + \rho^{-(1+\alpha)}C_0(1 + \rho|x|)\vspace{0.2cm}\\
&\leq& \rho^{(\bar{\alpha} - \alpha)}(5 +6 C_0)|x|^{1 +\bar{\alpha}}\vspace{0.2cm}\\
&\leq & |x|^{1+\bar{\alpha}}.
\end{array}
\]
where in the last inequality we used \eqref{rho_choice}. On the other hand, if $|x|\rho \leq 1/2$, we obtain
\[
\begin{array}{ccl}
|v_{m+1}(x)|&\leq& \rho^{-(1+\alpha)}(|v_{m}(\rho x) - \tilde{h}(\rho x)| + |\tilde{h}(\rho x) -\tilde{\ell}_m(\rho x)|\vspace{0.2cm}\\
&\leq &  \rho^{-(1+\alpha)}\left(\dfrac{\rho^{1+\alpha}}{2}+ C_0\rho^{1+ \bar{\alpha}}|x|^{1+ \bar{\alpha}}\right) \vspace{0.2cm}\\
&\leq & 1/2 + C_0 \rho^{(\bar{\alpha} - \alpha)}|x|^{1 + \bar{\alpha}}\\
&\leq & |x|^{1+\bar{\alpha}}.
\end{array}
\]
where $\tilde{h}$ comes from the Lemma \ref{approx_lemma2} and we used \eqref{rho_choice} again in the last inequality. This finishes the proof.
\end{proof}

Now we are ready to present the proof of Theorem \ref{Degenerate Theorem at least one solution}. 
 
\begin{proof}[Proof of the Theorem \ref{Degenerate Theorem at least one solution}.]
From \eqref{eq_06}, we conclude that the sequences $(a_j)_{j \in \mathbb{N}}$ and $(b_j)_{j \in \mathbb{N}}$ are Cauchy sequences. Hence, there exist constants $a_{\infty}$ and $b_{\infty}$ such that
\[
\lim_{j \rightarrow \infty}a_j = a_{\infty} \quad \mbox{and} \quad \lim_{j \rightarrow \infty}b_j = b_{\infty}.
\]
Moreover, we have the estimates
\[
|a_j - a_{\infty}| \leq  C {\rho}^{j(1+ \alpha)} \quad \mbox{and} \quad |b_j  - b_{\infty}|\leq C {\rho}^{j\alpha}.
\]
Fix $0< \rho \ll 1$, and let $k \in \mathbb{N}$ be such that $\rho^{k+1} < r < \rho^k$. We estimate from the previous computations
\[
\begin{array}{ccl}
\displaystyle\sup_{B_r} |u(x)- {\ell}_{\infty}(x)| & \leq & \displaystyle\sup_{B_{\rho^k}} |u(x) -{\ell}_k(x)| + \sup_{B_{\rho^k}} | {\ell}_k(x) - {\ell}_{\infty}(x)|\vspace{0.1cm}\\
                             & \leq & {\rho}^{k (1 + \alpha)} + |a_k - a_{\infty}| + {\rho}^k|b_k - b_{\infty}|\vspace{0.1cm}\\
                              & \leq & {\rho}^{k (1 + \alpha)} + C {\rho}^{k (1 + \alpha)} +  C {\rho}^k \cdot {\rho}^{k\alpha}\vspace{0.1cm}\\
                             & \leq & C \left(\frac{1}{\rho}\right)^{(1 + \alpha)} {\rho}^{(k+1)(1 + \alpha)} \vspace{0.1cm}\\
                              &\leq & C r^{(1 + \alpha)}. \\
\end{array}
\]

To conclude the proof, we characterize the coefficients $a_{\infty}$ and $b_{\infty}$. In fact, by taking the limit in the inequality
\begin{equation}
\sup_{B_{{\rho}^n}}|u(x) - {\ell}_n(x)| \leq C \rho^{n(1+ \alpha)},
\end{equation}
evaluated at $0$, we obtain that $a_{\infty} = u(0,0)$. We can also conclude that $b_{\infty}= Du(0,0)$ see for instance \cite{Bronzi-Pimentel-Rampasso-Teixeira2020}. We have proved that approximated viscosity solutions are of class $C^{1,\alpha}_{loc}$, and that they exist by Proposition \ref{prop_exis2}. Since approximated viscosity solutions are also viscosity solutions (Proposition \ref{equiv_def}), we have the existence of a $C^{1, \alpha}_{loc}(B_1)$ viscosity solution of \eqref{main_eq0}. 
\end{proof}

\subsection{$C^{1,\alpha}$ regularity via a smallness condition on $2-\sigma$}

In this subsection, we give a proof of Theorem \ref{Degenerate Theorem}. We start with an approximation lemma, which allows us to relate our model with a second order local equation.

\begin{Lemma}\label{approx_lemma}
Let $u \in C(\overline{B}_1)$ be a normalized viscosity solution to \eqref{eq_scaled}. Suppose that A\ref{assump_degree}-A\ref{assump_condition_on_kernels} are in force. Given $\mathfrak{M}, \delta>0$ and $\alpha \in (0, 1)$, there exists $\varepsilon>0$ such that if 
\begin{equation}\label{eq0_lemma}
|u(x)|\leq \mathfrak{M}\left(1 + |x|^{1 +\alpha}\right) \quad \text{for\;all}\; \; x \in \mathbb{R}^d
\end{equation}
and 
\[
|\sigma - 2| + \| f\|_{L^{\infty}(B_1)} \leq \varepsilon,
\]
then we can find a function $h \in {\mathcal C}_{loc}^{1, \bar{\alpha}}(B_1)$ with $\bar{\alpha} \in (0, 1)$ satisfying 
\[
\|u - h\|_{L^{\infty}(B_{3/4})} \leq \delta.
\]
\end{Lemma}
\begin{proof}
We argue by contradiction. That is, there are $\mathfrak{M}_0, \delta_0 >0, \alpha_0 \in (0, 1)$ and sequences $(\xi_j)_{j \in \mathbb{N}}$, $(a_j)_{j \in \mathbb{N}}$,  $(u_j)_{j \in \mathbb{N}}$,  $(f_j)_{j \in \mathbb{N}}$ and  $({\sigma}_j)_{j \in \mathbb{N}}$ such that
\begin{equation}\label{eq_approx}
-\left(|Du_j + \xi_j|^p + a_j(x)|Du_j + \xi_j|^q\right) {\mathcal I}_{\sigma_j}(u_j, x)= f_j(x) \quad \text{in} \quad B_1,
\end{equation}
\[
|u_j(x)|\leq \mathfrak{M}_0\left(1 + |x|^{1 +\alpha_0}\right),
\]
and 
\[
|\sigma_j - 2| + \| f_j\|_{L^{\infty}(B_1)} \leq 1/j,
\]
however
\begin{equation}\label{contradiction_argument}
\|u_j - h\|_{L^{\infty}(B_{3/4})} \geq \delta_0
\end{equation}
for all $h \in {\mathcal C}_{loc}^{1, \bar{\alpha}}(B_1)$.

Since $\sigma_j \rightarrow 2$, thanks to $A\ref{assump_condition_on_kernels}$ we have $\mathcal{I}_{\sigma_j} \rightarrow F_{\infty}$, where $F_{\infty}$ is a uniformly elliptic operator. In addition,  from Proposition \ref{Lip_reg} we can guarantee the existence of a function $u_\infty$ such that $u_j \rightarrow  u_{\infty}$ in $B_{3/4}$. We shall prove that $u_{\infty}$ solves 
\begin{equation}\label{eq_05}
F_{\infty}(D^2 u_{\infty})=0,
\end{equation}
in the viscosity sense. 

We first show that $u_{\infty}$ is a viscosity supersolution. Let $\varphi \in \mathcal{C}^2(\mathbb{R}^d)$ be a test function touching $u_{\infty}$ strictly from below at $\tilde{x} \in B_1$, \emph{ i.e.}, $\varphi(\tilde{x}) = u_{\infty}(\tilde{x})$ and $\varphi(x)< u_{\infty}(x)$ in a small neighborhood of $\tilde{x}$. For simplicity, we can take $\tilde{x} =0$ so that $\varphi(0)=u_{\infty}(0)=0$.

We can assume that $\varphi$ is given by the quadratic polynomial
\[
\varphi(x):=\frac{1}{2}M x\cdot x + b\cdot x.
\]
Since $u_j \rightarrow u_{\infty}$ uniformly in $B_{3/4}$, we may find a point $x_j$ and a quadratic polynomial of the form 
\[
\varphi_j(x):=\dfrac{1}{2} M(x - x_j)\cdot(x-x_j)+ b\cdot(x-x_j)+ u_j(x_j)
\]
such that $\varphi_j$ touches $u_j$ from below at $x_j$ in a small neighborhood of the origin. Since $u_j$ solves \eqref{eq_approx}, we obtain that
\[
-\left(|b+\xi_j|^p +a_j(x_j)|b+\xi_j|^q\right) \mathcal{I}_{\delta}(u_j, \varphi_j, x_j) \geq f_j(x_j).
\]

Next, we investigate two cases. We start by considering the case in which the sequence $(\xi_j)_{j \in {\mathbb N}}$ is unbounded. In this case, we can find a subsequence $\xi_j$ such that $|\xi_j| \to \infty$ as $j\to\infty$. Let $j \in \mathbb{N}$ be sufficiently large such that $|{\xi}_j|\geq \max\left \{1, 2|b|\right\}$. First, notice that
\[
\frac{1}{|b+ \xi_j|^p}\leq \frac{2^p}{|\xi_j|^p}  \longrightarrow 0 \quad \text{as} \quad j\rightarrow\infty,
\] 
which implies,
\[
\left|\dfrac{f_j(x_j)}{|b+\xi_j|^p +a_j(x)|b+\xi_j|^q}\right|\leq \dfrac{2^p}{j|\xi_j|^p} \longrightarrow 0 \quad \text{as} \quad j\rightarrow\infty.
\]
Thus, by passing the limit in \eqref{eq_approx}, we conclude that
\[
- F_{\infty}(M) = - \lim_{j \to\infty} \mathcal{I}_{\delta}(u_j, \varphi_j, x_j)  \geq -\lim_{j\to\infty} \dfrac{2^p}{j|\xi_j|^p} =0.
\]

On the other hand, if the sequence $(\xi_j)_{j \in \mathbb{N}}$ is bounded, then up
to a subsequence, $\xi_j$ converges to $\xi_{\infty}$. Hence $\xi_j + b$ converges to $\xi_{\infty} +b$ as $j$ goes to infinity. Observe that, if $|\xi_{\infty} +b|>0$, then $|\xi_j +b|\geq \frac{1}{2}|\xi_{\infty} +b|>0$ for $j$ sufficiently large. Hence, we obtain
\[
- F_{\infty}(M) = - \lim_{j \rightarrow\infty} \mathcal{I}_{\delta}(u_j, \varphi_j, x_j) \geq -\lim_{j \rightarrow\infty} \dfrac{2^p}{j|\xi_{\infty} +b|^p} =0.
\]

In the sequel, we analyze the case $|\xi_{\infty} + b|=0$. Suppose by contradiction that 
\begin{equation}\label{eq_02}
-F_{\infty}(M)<0.
\end{equation}
Since $F_\infty$ is uniformly elliptic, we have that $M$ has at least one positive eigenvalue. Define $\Gamma$ as the direct sum of all the eigensubspaces generated by the nonnegative eigenvalues of $M$ and let $\Pi_{\Gamma}$ be the orthogonal projection over $\Gamma$. Since $|\xi_{\infty} + b|=0$, either $\xi_{\infty}= -b$ with $|\xi_{\infty}|, |b|>0$ or $|\xi_{\infty}|=|b|=0$.

We first analyze the case $\xi_{\infty}= -b$ with $|\xi_{\infty}|, |b|>0$. Recall that $\varphi$ touches $u_{\infty}$ at $0$ from below, then for $k>0$ sufficiently small, the function
\[
\phi_{k}(x):= \dfrac{1}{2}Mx\cdot x + b\cdot x + k|\Pi_{\Gamma}(x)|
\]
touches $u_j$ from below at $\tilde{x}_j$ in a neighborhood of the origin. Since  the sequence $\{\tilde{x}_j\}$ is bounded, up to a subsequence, we obtain that $\tilde{x}_j\rightarrow \tilde{x}_{\infty}$ for some $ \tilde{x}_{\infty} \in B_1$. Now, we study two cases: $\Pi_{\Gamma}(\tilde{x}_{\infty})=0$ and $\Pi_{\Gamma}(\tilde{x}_{\infty})\not=0$. In the case $\Pi_{\Gamma}(\tilde{x}_{\infty})=0$, we have
\[
 |\Pi_{\Gamma}(\tilde{x}_{\infty})|= \max_{e \in \mathbb{S}^{d-1}} e \cdot \Pi_{\Gamma}(\tilde{x}_{\infty})=\min_{e \in \mathbb{S}^{d-1}} e \cdot \Pi_{\Gamma}(\tilde{x}_{\infty}),
\]
consequently, we can rewrite the function $\phi_{k}$ as 
\[
\phi_{k}(x):= \dfrac{1}{2}Mx\cdot x + b\cdot x + k e\cdot\Pi_{\Gamma}(x) \quad \text{for \; all} \quad e \in  \mathbb{S}^{d-1}.
\]
Notice that 
\[
D\left(e\cdot\Pi_{\Gamma}(x)\right) = \Pi_{\Gamma}(e), \;\mbox{ and }\; D^2(e \cdot \Pi_{\Gamma}(x)) =0. 
\]
In addition,
\begin{equation}\label{eq_03}
e \in \mathbb{S}^{d-1}\cap \Gamma \Rightarrow \Pi_{\Gamma}(e) =e \quad \text{and} \quad e \in \mathbb{S}^{d-1}\cap \Gamma^{\bot} \Rightarrow \Pi_{\Gamma}(e) =0.
\end{equation}
If $|M\tilde{x}_{\infty}| =0$, then $|M\tilde{x}_j|+ |b +\xi_j| \leq k/2$ for every $j$ large enough, so that $|M\tilde{x}_j + b +\xi_j + k e|\geq k/2$. Therefore
{
\small
\[
\begin{array}{ccl}
-F_{\infty}(M) &=& -\lim_{j \rightarrow \infty}\mathcal{I}_{\delta}(u_j, \varphi_j, \tilde{x}_j)\vspace{0.3cm}\\
&\geq& -\lim_{j \rightarrow \infty} \left|\dfrac{f_j(\tilde{x}_j)}{|M\tilde{x}_j + b + \xi_j +k e|^p + a_j(\tilde{x}_j)|M\tilde{x}_j + b + \xi_j +k e|^q} \right|\vspace{0.3cm}\\ 
&\geq& - \lim_{j \rightarrow \infty} \dfrac{2^p}{j k^p}
\vspace{0.3cm}\\ 
&=&0. \\ 
\end{array}
\]
}
This implies that $-F_{\infty}(M)\geq 0$, which is a contradiction with \eqref{eq_02}.

Now suppose that $|M\tilde{x}_{\infty}|>0$ and $\Gamma \equiv \mathbb{R}^d$. By using \eqref{eq_03}, we take $e \in \mathbb{S}^{d-1}$ such that 
\[
|M\tilde{x}_{\infty} +k \Pi_{\Gamma}(e)| =|M\tilde{x}_{\infty} +k e|>0.
\] 
Hence, for all $j$ sufficiently large, we get 
\[
|M\tilde{x}_j +k e|\geq \frac{1}{2}|M\tilde{x}_{\infty} +k e|>0 \;\mbox{and} \; |b + \xi_j| \leq \frac{1}{4}|M \tilde{x}_{\infty} +k e|.
\]
If $\Gamma \not\equiv \mathbb{R}^d$, then we can find $e \in \mathbb{S}^{d-1}\cap \Gamma^{\bot}$ for which by \eqref{eq_03} it holds
\[
|M\tilde{x}_{\infty} +k \Pi_{\Gamma}(e)| =|Mx_{\infty}|>0.
\]
Thus,  for all $j$ large enough, we have 
\[
|M\tilde{x}_j|\geq \frac{1}{2}|M\tilde{x}_{\infty}| \;\mbox{and} \; |b + \xi_j| \leq \frac{1}{4}|M \tilde{x}_{\infty}|.
\]
In both cases, we obtain that
\[
|M\tilde{x}_j  +b +\xi_j + k \Pi_{\Gamma}(e)| \geq \dfrac{1}{4}\left|M\tilde{x}_{\infty} + k\Pi_{\Gamma}(e)\right|>0.
\]
Therefore,
{\footnotesize
\[
\begin{array}{ccl}
-F_{\infty}(M)&=&- \lim_{j \rightarrow \infty} \mathcal{I}_{\delta}(u_j, \varphi_j, \tilde{x}_j)\vspace{0.3cm}\\
&\geq& - \lim_{j \rightarrow \infty} \left|\dfrac{f_j(\tilde{x}_j)}{\left|M\tilde{x}_j +b + \xi_j +k\Pi_{\Gamma}(e) \right|^p + a_j(\tilde{x}_j)\left|M\tilde{x}_j +b + \xi_j +k\Pi_{\Gamma}(e) \right|^q}\right|\vspace{0.3cm}\\
&\geq& - \lim_{j \rightarrow \infty} \dfrac{2^{2p}}{j |M\tilde{x}_{\infty} +k\Pi_{\Gamma}(e)|}\vspace{0.3cm}\\
&=& 0,\\ 
\end{array}
\]
}
which is again a contradiction with \eqref{eq_02}. 

Finally, let us consider the case $|\Pi_{\Gamma}(\tilde{x}_{\infty})|>0$. As before, for $j$ sufficiently large, we get that $\Pi_{\Gamma}(\tilde{x}_j)>0$. It follows that the application $x \longrightarrow |\Pi_{\Gamma}(x)|$ is smooth and convex in a neighborhood of $\tilde{x}_j$. Since $\Pi_{\Gamma}$ is a projection, we have 
\begin{equation}\label{eq_04}
\Pi_{\Gamma}(x)D(\Pi_{\Gamma}(x))= \Pi_{\Gamma}(x) \;\mbox{ and } \;D^2\left(\Pi_{\Gamma}(x)\right)\geq0. 
\end{equation}
Hence,
{\small
\[
\mathcal{I}_{\delta}(u_j, {\phi}_{k}, \tilde{x}_j)\geq \dfrac{f_j(\tilde{x}_j)}{\left|M \tilde{x}_j + b+\xi_j +k\frac{\Pi_{\Gamma}(\tilde{x}_j)}{|\Pi_{\Gamma}(\tilde{x}_j)|} \right|^p + a_j(\tilde{x}_j)\left|M\tilde{x}_j + b +\xi_j +k \frac{\Pi_{\Gamma}(\tilde{x}_j)}{|\Pi_{\Gamma}(\tilde{x}_j)|} \right|^q }.
\]
}
We continue as before with $e=e_j:= \frac{\Pi_{\Gamma}(\tilde{x}_j)}{|\Pi_{\Gamma}(\tilde{x}_j)|}$, that is, studying the cases $|M\tilde{x}_{\infty}| =0$ and $|M\tilde{x}_{\infty}| >0$, to obtain respectively
 \[\mathcal{I}_{\delta}(u_j, {\phi}_{k}, \tilde{x}_j)\geq-\frac{2^p}{j k^p}\quad \text{and} \quad \mathcal{I}_{\delta}(u_j, {\phi}_{k}, \tilde{x}_j)\geq-\frac{2^{2p}}{j|Mx_{\infty} + k\Pi_{\Gamma}(e)|^p}.
\]
Therefore, we can conclude $-F_{\infty}(M + D^2\left(\Pi_{\Gamma}(\tilde{x}_{\infty})\right))\geq 0$, which implies by \eqref{eq_04} that $- F_{\infty}(M)\geq0$. Hence, we reach  a contradiction with \eqref{eq_02}. 

To complete the proof, we need to analyze the case $|\xi_{\infty}|=|b|=0$. We proceed as before, using now the function 
\[
\phi_{k}(x):=\dfrac{1}{2}Mx\cdot x +k|\Pi_{\Gamma}(x)|,
\]
that touches $u_j$ from below at $\tilde{x}_j$ in a neighborhood at the origin, arriving in a contradiction. 

Therefore, we conclude that $u_{\infty}$ is a viscosity supersolution to \eqref{eq_05}.
Similarly, we can prove that $u_{\infty}$ is also a subsolution to \eqref{eq_05}, which implies that $u_{\infty}$ is a solution. Since the solution to \eqref{eq_05} are locally of class $C^{1,\bar{\alpha}}$, we can take $h = u_{\infty}$, and we reach a contradiction with \eqref{contradiction_argument}. This concludes the proof.
\end{proof}

Now we proceed with the proof of Theorem \ref{Degenerate Theorem}. 
 
\begin{proof}[Proof of the Theorem \ref{Degenerate Theorem}]
Once we have the Lemma \ref{approx_lemma}, we can prove the same results as in Proposition \ref{Deg_Step1} and Propostition \ref{Deg_step2} with the same hypotheses of Lemma \ref{approx_lemma} and then the proof of Theorem \ref{Degenerate Theorem} follows exactly as in the proof of Theorem  \ref{Degenerate Theorem at least one solution}.
\end{proof}

\section{Analysis of the singular case}

In this section, we deal with the singular case $-1 < p \leq q <0$.  For the sake of clarity, we split the section into two subsections: in the first one, we prove  the existence of approximated viscosity solutions and in the second one, we establish the regularity results.

\subsection{Existence of approximated viscosity solutions}

For each $j \in \mathbb{N}$, we consider the nonlocal uniformly elliptic equation
\begin{equation}\label{eq_reg}
-\left((|Du_j|+c_j)^p+a(x)(|Du_j|+c_j)^q\right){\cal I}_{\sigma}(u_j,x) = f(x) \;\;\mbox{ in }\;\;B_1,
\end{equation}
where $(c_j)_{j \in \mathbb{N}}$ is a sequence of positive real numbers such that $c_j \rightarrow 0$ and $c_j \leq 1$ for every $j$. We prove that the sequence $(u_j)_{j \in \mathbb{N}}$ of viscosity solutions for \eqref{eq_reg}, with boundary data in $ L^{1}_{\sigma}(\mathbb{R}^d)$, converge to a function $u_\infty\in C(\overline{B}_1)\cap L^{1}_{\sigma}(\mathbb{R}^d)$ which is an approximated viscosity solution to equation \eqref{main_eq0}. 

We start with a compactness result, independent of $j$, for the solutions of \eqref{eq_reg}. 

\begin{Lemma}\label{lip_reg_sing}
Let  $u_j \in C(\overline{B}_1)$ be a viscosity solution to \eqref{eq_reg}. Suppose that A\ref{assump_degree}-A\ref{assump_regularf} are in force. Then $u_j$ is locally Lipschitz continuous,\emph{ i. e.,}
\begin{equation*}
|u_j(x) - u_j(y)| \leq M|x - y| \quad  \text{for all}\: \: x, y  \in B_{1/2},
\end{equation*}
where $M>0$ does not depend on j.
\end{Lemma}

\begin{proof} 
The proof follows the same lines as in Lemma \ref{Lipschitz_Regularity2}, hence we will only provide a few details. Consider $\varphi$ and $\psi$ as in Lemma \ref{Lipschitz_Regularity2}. As before, we define $\phi, \Phi:\mathbb{R}^d \times \mathbb{R}^d \rightarrow \mathbb{R}$ as
\[
\phi(x, y):= L\varphi(|x- y|) + \psi(y)
\]
and 
\[
\Phi(x, y) := u_j(x) - u_j(y) - \phi(x, y).
\]
Since $\Phi$ is a continuous function, we have $\Phi$ attains its maximum in $\overline{B}_1 \times \overline{B}_1$ at $(\bar{x}, \bar{y})$. We argue through a contradiction argument, suppose that 
\[
\Phi(\bar{x}, \bar{y}) > 0,
\]
and with the same notation as Lemma \ref{Lipschitz_Regularity2}, we obtain 
\[
- \left((|{\xi}_{\bar{x}}| + c_j)^p + a(\bar{x})(|{\xi}_{\bar{x}}|+c_j)^q \right){\mathcal I}_{\delta}(u_j, {\Phi}_{\bar{y}}, \bar{x})\leq f(\bar{x}).
\]
Notice that
\[
|{\xi}_{\bar{x}}| \leq L(2+\alpha),
\]
which yields to
\[
\begin{array}{ccl}
- {\mathcal I}_{\delta}(u_j, {\Phi}_{\bar{y}}, \bar{x})
&\leq & \dfrac{\|f\|_{L^{\infty}(B_1)}}{(|{\xi}_{\bar{x}}| + c_j)^p + a(\bar{x})(|{\xi}_{\bar{x}}|+c_j)^q} \vspace{0.3cm} \\
&\leq&\dfrac{\|f\|_{L^{\infty}(B_1)}}{(|{\xi}_{\bar{x}}| + c_j)^p} \vspace{0.3cm} \\
&\leq& (L(2+\alpha) + 1)^{-p}\|f\|_{L^{\infty}(B_1)}\vspace{0.3cm} \\
&\leq& C_1\|f\|_{L^{\infty}(B_1)}\vspace{0.3cm}.
\end{array}
\]
An analogous reasoning yields $-\mathcal{I}_{\delta}(u_j, - \Phi_{\bar{x}}, \bar{y}) \geq -C_2\|f\|_{L^{\infty}(B_1)}$, where $C_1$ and $C_2$ are universal constants and consequently
\[
{\mathcal I}_{\delta}({u_j, \Phi}_{\bar{y}}, \bar{x}) - \mathcal{I}_{\delta}(u_j, - \Phi_{\bar{x}}, \bar{y})\geq -  2\bar{C}\|f\|_{L^{\infty}(B_1)},
\]
where $\bar{C}$ is a universal constant.

From here, we can proceed exactly as in Proposition \ref{Lipschitz_Regularity2} to finish the proof.
\end{proof}

In the next result, we prove existence of an approximated viscosity solutions for the equation \eqref{main_eq0}.

\begin{Proposition}\label{prop_exis}
Suppose that A\ref{assump_degree}-A\ref{assump_regularf} hold true. Then, there exists at least one approximated viscosity solution $u\in C(B_1)$ to \eqref{main_eq0}.
\end{Proposition}

\begin{proof}
Consider the equation
\[
\left\{
\begin{array}{rcl}
\left[(|Du_j|+1/j)^p+a(x)(|Du_j|+1/j)^q\right]{\mathcal I}_{\sigma}(u_j,x) & = &f(x) \;\;\mbox{ in }\;\;B_1 \\
u_j(x) & = & g(x)\;\;\mbox{ in }\;\;\mathbb{R}^d\setminus B_1,
\end{array}
\right.
\]
where $g \in L^1_{\sigma}(\mathbb{R}^d)$  satisfies 
\[
g(x) \leq 1 + |x|^{1+\alpha}\;\;\mbox{ for all }\;\; x\in \mathbb{R}^d,
\]
where $1+\alpha \in(0,\sigma)$. Since the operator is nonlocal uniformly elliptic, the existence of a viscosity solution $u_j$ is assured by \cite{Barles-Chasseigne-Imbert-2008}. By Lemma \ref{lip_reg_sing}, we have that the solution $u_j \in C^{0,1}$, with estimates that are independent of $j$. Hence, there exists $u_\infty \in C_{loc}(B_1)$  such that $u_j \to u_{\infty}$ locally uniformly in $B_1$, and $u_{\infty} \equiv g$ outside $B_1$. By taking $c_j := 1/j$, we have directly by definition that $u_\infty$ is an approximated viscosity solution to \eqref{main_eq0}.
\end{proof}

\subsection{$C^{1,\alpha}$-regularity estimates} 

This subsection is devoted to the proof of Theorem \ref{Singular Theorem}. We start with a compactness result to approximated viscosity solutions of \eqref{main_eq0}. 

%
Notice that the operator in \eqref{eq_reg} is a nonlocal uniformly elliptic operator, hence for each $j \in \mathbb{N}$, we have  $u_j \in C^{1,\alpha}_{loc}(B_1)$. Keep in mind that the $C^{1,\alpha}$-norm of $u_j$ could degenerate as $j \to \infty$. In what follows, we will verify that this does not happen. 

\begin{Proposition}\label{prop_ur}
Let $u_j \in C(\overline{B}_1)$ be a viscosity solution to \eqref{eq_reg}. Suppose that A1 - A4 are in force. Then $u_j \in C^{1,\alpha}_{loc}(B_1)$ with $\alpha \in (0, 1)$. In addition, there exists a positive constant $C=C(\lambda, \Lambda, d,p)$ such that
\[
\|u_j\|_{C^{1,\alpha}(B_{1/2})} \leq C(\|u_j\|_{L^{\infty}(B_1)} + \|u_j\|_{L^1_ \sigma} +\|f\|_{L^{\infty}(B_1)}).
\]
\end{Proposition}

\begin{proof}
Let $\varphi_j \in C^2(\mathbb{R}^d)$ be a test function that touches $u_j$ from below at $x_j$. We write the viscosity inequality 
\[
-\left[(|D\varphi_j(x_j)| + c_j)^p +a(x_j)(|D\varphi_j(x_j)| + c_j)^q\right]{\mathcal I}_{\delta}(u_j, \varphi_j, x_j) \leq f(x_j).
\]
It implies that 
\[
-{\mathcal I}_{\delta}(u_j, \varphi_j, x_j) \leq (|D\varphi_j(x_j)|+c_j)^{-p}\|f\|_{L^{\infty}(B_1)} \leq (|D\varphi_j(x_j)|+1)^{-p}\|f\|_{L^{\infty}(B_1)}.
\]
Since $u_j$ is of class $C^{1,\alpha}_{loc}(B_1)$ and $\varphi_j$ touches $u_j$ from below at $x_j$, we can conclude from Lemma  \ref{lip_reg_sing} that $|D\varphi_j(x_j)| = |Du_j(x_j)| \leq M$. Hence
\[
-{\mathcal I}_{\sigma}(u_j,x) \leq M^{-p}\|f\|_{L^{\infty}(B_1)}.
\] 
Similarly, we can show that
\[
-{\mathcal I}_{\sigma}(u_j,x) \geq -M^{-p}\|f\|_{L^{\infty}(B_1)}.
\]
Therefore by \cite[Theorem 52]{Caffarelli-Silvestre2011} we have that $u_j \in C^{1,\alpha}_{loc}(B_1)$ and
\[
\|u_j\|_{C^{1,\alpha}(B_{1/2})} \leq C(\|u_j\|_{L^{\infty}(B_1)} + \|u_j\|_{L^1_ \sigma}+ \|f\|_{L^{\infty}(B_1)}),
\] 
where $C>0$ is a universal constant.
\end{proof}

At this point, we are able to prove the Theorem \ref{Singular Theorem}.

\begin{proof}[Proof of Theorem \ref{Singular Theorem}]
From Definition \ref{new_def}, there exist  sequences $(u_j)_{j\in \mathbb{N}} \in C(\overline{B}_1)\cap L^{1}_{\sigma}(\mathbb{R}^d)$, and $(c_j)_{j\in \mathbb{N}} \in \mathbb{R}^+$ fulfilling $ u_j$ converges locally uniformly to $u$ in  $ B_1$ and $c_j \to 0$, such that $u_j$ is a viscosity solution to
\[
-\left((|Du_j|+c_j)^p + a(x)(|Du_j|+c_j)^q\right){\mathcal I}_\sigma(u_j,x) =  f \;\;\mbox{ in }\;\;B_1.
\]

It follows from Proposition \ref{prop_ur} that $u_j$ is of class $C^{1, \alpha}_{loc}$ with estimates
\[
\|u_j\|_{C^{1,\alpha}(B_{1/2})} \leq C(\|u_j\|_{L^{\infty}(B_1)} + \|f\|_{L^{\infty}(B_1)}).
\]
where $C>0$ is a constant that does not depend on $j$.

By applying the limit in the estimate above as $j \rightarrow \infty$ combined with the fact that $u_j$ converges locally uniformly to $u$ in $B_1$, we obtain the following estimate
\[
\|u\|_{C^{1,\alpha}(B_{1/2})} \leq C(\|u\|_{L^{\infty}(B_1)} + \|u\|_{L^1_ \sigma} + \|f\|_{L^{\infty}(B_1)}).
\]
This completes the proof of the theorem.
\end{proof}

\bigskip
{\bf Acknowledgement:} PA was partially supported by  CAPES-INCTMat - Brazil and by the Portuguese government through FCT-Funda\c c\~ao para a Ci\^encia e a Tecnologia, I.P., under the project UID/MAT/04459/2020. D dos P was partially supported by CNPq and CAPES/Fapitec. MS was partially supported by FAPESP grant 2021/04524-0 and by the Portuguese government through FCT-Funda\c c\~ao para a Ci\^encia e a Tecnologia, I.P., under the projects UID/MAT/04459/2020, and PTDC/MAT-PUR/1788/2020. This study was financed in part by the Coordena\c c\~ao de Aperfei\c coamento de Pessoal de N\'ivel Superior - Brazil (CAPES) - Finance Code 001.

\bibliography{andrade_prazeres_santos}

\bibliographystyle{plain}

\bigskip

\noindent\textsc{P\^edra D. S. Andrade}\\
Instituto Superior T\'ecnico\\
Universidade de Lisboa -- ULisboa\\
 1049-001,  Av. Rovisco Pais, Lisboa, Portugal\\
\noindent\texttt{pedra.andrade@tecnico.ulisboa.pt}

\vspace{.15in}

\noindent\textsc{Disson S. dos Prazeres}\\
Departamento de Matem\'atica\\
Universidade Federal de Sergipe -- UFS\\
49100-000, Roza Elze, S\~ao Cristov\~ao - SE, Brazil\\
\noindent\texttt{disson.mat.ufs.br}

\vspace{.15in}

\noindent\textsc{Makson S. Santos}\\
Instituto Superior T\'ecnico\\
Universidade de Lisboa -- ULisboa\\
 1049-001,  Av. Rovisco Pais, Lisboa, Portugal\\
\noindent\texttt{makson.santos@tecnico.ulisboa.pt}


\end{document}